\documentclass[12pt, reqno]{amsart}

\usepackage{amsmath, amsthm, amscd, amsfonts, amssymb, graphicx, color}
\usepackage[bookmarksnumbered, colorlinks, plainpages]{hyperref}
\input{mathrsfs.sty}
\hypersetup{colorlinks=true,linkcolor=red, anchorcolor=green, citecolor=cyan, urlcolor=red, filecolor=magenta, pdftoolbar=true}

\newtheorem{theorem}{Theorem}[section]
\newtheorem{lemma}[theorem]{Lemma}
\newtheorem{proposition}[theorem]{Proposition}
\newtheorem{corollary}[theorem]{Corollary}
\theoremstyle{definition}

\newtheorem{example}[theorem]{Example}

\theoremstyle{remark}
\newtheorem{remark}[theorem]{Remark}
\numberwithin{equation}{section}

\begin{document}

\title{Angle Preserving Mappings}

\author[M.S. Moslehian, A. Zamani, M. Frank]{Mohammad Sal Moslehian, Ali Zamani and Michael Frank}

\address{$^1$ Department of Pure Mathematics, Center of Excellence in
Analysis on Algebraic Structures (CEAAS), Ferdowsi University of
Mashhad, P.O. Box 1159, Mashhad 91775, Iran}
\email{moslehian@um.ac.ir, moslehian@member.ams.org}
\email{Zamani.ali85@yahoo.com}
\address{$^2$ Hochschule f\"ur Technik, Wirtschaft und Kultur (HTWK) Leipzig, Fakult\"at Informatik, Mathematik und Naturwissenschaften, PF 301166, 04251 Leipzig, Germany}
\email{mfrank@imn.htwk-leipzig.de}

\begin{abstract}
In this paper, we give some characterizations of orthogonality preserving mappings between inner product spaces. Furthermore, we study the linear mappings that preserve some angles. One of our main results states that if $\mathcal{X}, \mathcal{Y}$ are real inner product spaces and $\theta\in(0, \pi)$, then an injective nonzero linear mapping $T:\mathcal{X}\longrightarrow \mathcal{Y}$ is a similarity whenever (i) $x\underset{\theta}{\angle} y\, \Leftrightarrow \,Tx\underset{\theta}{\angle} Ty$ for all $x, y\in \mathcal{X}$; (ii) for all $x, y\in \mathcal{X}$, $\|x\|=\|y\|$ and $x\underset{\theta}{\angle} y$ ensure that $\|Tx\|=\|Ty\|$. We also investigate orthogonality preserving mappings in the setting of inner product $C^{*}$-modules. Another result shows that if $\mathbb{K}(\mathscr{H})\subseteq\mathscr{A}\subseteq\mathbb{B}(\mathscr{H})$ is a $C^{*}$-algebra and $T\,:\mathscr{E}\longrightarrow \mathscr{F}$ is an $\mathscr{A}$-linear mapping between inner product
$\mathscr{A}$-modules, then $T$ is orthogonality preserving if and only if $|x|\leq|y|\, \Rightarrow \,|Tx|\leq|Ty|$ for all $x, y\in \mathscr{E}$.
\end{abstract} \maketitle

\section{Introduction}

Let $(\mathcal{X}, \langle \cdot, \cdot\rangle)$ be a real inner product space and $x, y\in \mathcal{X}\smallsetminus\{0\}$. From the Cauchy--Schwarz inequality, we have $-1\leq\frac{\langle x, y\rangle}{\|x\|\|y\|}\leq1$. Therefore, there is a unique number $\widehat{(x, y)}\in[0, \pi]$ such that $\widehat{(x, y)}=\arccos\frac{\langle x, y\rangle}{\|x\|\|y\|}$. The number $\widehat{(x, y)}$ is called the angle between $x$ and $y$. If $\theta=\widehat{(x, y)}$, we write $x\underset{\theta}{\angle} y$. If $\widehat{(x, y)}=0$ or $\pi$, we say that $x, y$ are parallel and we write $x\parallel y$ (see \cite{M.Z}). Also, if $\widehat{(x, y)}=\frac{\pi}{2}$, we say that $x, y$ are orthogonal and we write $x\perp y$.\\
We call a mapping $T:\mathcal{X}\longrightarrow \mathcal{Y}$, where $\mathcal{X}$ and $\mathcal{Y}$ are inner product spaces, orthogonality preserving if $x\perp y\, \Rightarrow \,Tx\perp Ty$ for all $x, y\in \mathcal{X}$. Chmieli\'{n}ski \cite{Ch.1} proved that a linear mapping $T$ preserves orthogonality if and only if $T$ is an isometry multiplied by a positive constant. Blanco and Turn\v{s}ek \cite{B.T}, extended this result to the case of the linear mappings between normed spaces with the Birkhoff--James orthogonality. Recently the linear mappings preserving approximately Roberts orthogonality in normed spaces have been studied; cf. \cite{Z.M}. Further, Chmieli\'{n}ski \cite{Ch.2} studied stability of angle-preserving mappings on the plane.
In Section $3$ we give a new characterization of orthogonality preserving mappings between two real inner product spaces and investigate the linear mappings that preserve some angles.\\

It is natural to explore the orthogonality preserving mappings between inner product $C^{*}$-modules. Ili\v{s}evi\'{c} and Turn\v{s}ek \cite{I.T} studied approximately orthogonality preserving mappings on $C^{*}$-modules. Frank et al. \cite{F.M.P} extended some of the results of \cite{I.T}. Orthogonality preserving mappings have been treated also in \cite{G.I.M.S, Ch.3}. Burgos \cite{Bu} studied orthogonality preserving linear mappings on Hilbert $C^{*}$-algebras with non-zero socles. We have also to mention the survey of Leung et al. \cite{L.N.W.1}.
In Section $4$ we study orthogonality preserving mappings in inner product $C^{*}$-modules, in particular, we treat the local mappings.

\section{Preliminaries}

Let us fix our notation and terminology. A $C^*$-algebra is a complex Banach $\ast$-algebra $(\mathscr{A}, \|.\|)$ with the additional norm condition $\|a^\ast a\|=\|a\|^2$. The basic examples of $C^*$-algebras are
$\mathbb{K}(\mathscr{H})$ and $\mathbb{B}(\mathscr{H})$, the algebra of all compact operators and the algebra of all bounded operators on a complex Hilbert space $\mathscr{H}$, respectively. An element $a\in\mathscr{A}$ is called positive, denoted by $a\geq0$, if $a=b^* b$ for some $b\in\mathscr A$. If $0\leq a\leq b$, then $0\leq c^\ast ac\leq c^\ast bc\leq\|b\|c^\ast c$ for all $c\in\mathscr{A}$. In addition, if $a, b\geq0$ and $\|ac\|=\|bc\|$ for all $c\in\mathscr{A}$, then $a=b$. If $a\in\mathscr{A}$ is positive, then there exists a unique positive $b\in\mathscr{A}$ such that $a=b^2$. Such an element $b$ is called the positive square root of $a$ and denoted by $a^{\frac{1}{2}}$. Also, $0\leq a\leq b$ implies $0\leq a^{\frac{1}{2}} \leq b^{\frac{1}{2}}$. The converse does not hold in general, but it is valid in the commutative $C^*$-algebras. An approximate unit for a $C^*$-algebra $\mathscr{A}$ is an increasing net $(e_i)_{i\in I}$ of positive elements in the closed unit ball of $\mathscr{A}$ such that $\lim\limits_{i}\|a - ae_i\| = 0$ for all $a\in\mathscr{A}$. More details on the theory $C^*$-algebras can be found in \cite{mor}.\\

The notion of an inner product $C^{*}$-module is a natural generalization of that of an inner product space arising under the replacement of the field of scalars $\mathbb{C}$ by a $C^{*}$-algebra. Let $\mathscr{A}$ be a $C^*$-algebra. An inner product $\mathscr{A}$-module, or an inner product $C^*$-module over $\mathscr{A}$ is a complex linear space $\mathscr{E}$ which is a right $\mathscr{A}$-module with a compatible scalar multiplication (i.e., $\mu(xa) = (\mu x)a = x(\mu a)$ for all $x\in \mathscr{E}, a\in\mathscr{A}, \mu\in\mathbb{C}$) and equipped with an $\mathscr{A}$-valued inner product $\langle\cdot,\cdot\rangle : \mathscr{E}\times \mathscr{E}\longrightarrow\mathscr{A}$ satisfying\\
(i) $\langle x, \alpha y+\beta z\rangle=\alpha\langle x, y\rangle+\beta\langle x, z\rangle$,\\
(ii) $\langle x, ya\rangle=\langle x, y\rangle a$,\\
(iii) $\langle x, y\rangle^*=\langle y, x\rangle$,\\
(iv) $\langle x, x\rangle\geq0$ and $\langle x, x\rangle=0$ if and only if $x=0$,\\
for all $x, y, z\in \mathscr{E}, a\in\mathscr{A}, \alpha, \beta\in\mathbb{C}$.\\
A mapping $T:\mathscr{E}\longrightarrow \mathscr{F}$, where $\mathscr{E}$ and $\mathscr{F}$ are inner product $\mathscr{A}$-modules, is called $\mathscr{A}$-linear if it is linear and $T(xa)=(Tx)a$ for all $x\in \mathscr{E}$, $a\in\mathscr{A}$. For an inner product $\mathscr{A}$-module $\mathscr{E}$ the Cauchy--Schwarz inequality holds (see also \cite{ABFM}):
$$\|\langle x, y\rangle\|^2\leq\|\langle x, x\rangle\|\,\|\langle y, y\rangle\|\qquad(x, y\in \mathscr{E}).$$
Consequently, $\|x\|=\|\langle x, x\rangle\|^{\frac{1}{2}}$ defines a norm on $\mathscr{E}$. If $\mathscr{E}$ is complete with respect to this norm, then it is called a Hilbert $\mathscr{A}$-module, or a Hilbert $C^*$-module over $\mathscr{A}$. Complex Hilbert spaces are Hilbert $\mathbb{C}$-modules. Any $C^*$-algebra $\mathscr{A}$ can be regarded as a Hilbert $C^*$-module over itself via $\langle a, b\rangle :=a^* b$. For every $x\in \mathscr{E}$ the positive square root of $\langle x, x\rangle$ is denoted by $|x|$. In the case of a $C^*$-algebra we get the usual modulus of $a$, that is $|a|=(a^* a)^{\frac{1}{2}}$. Although the definition of $|x|$ has the same form as that of the norm of elements of inner product spaces, there are some significant differences. For instance, it does not satisfy the triangle inequality in general. Note that the theory of inner product $C^{*}$-modules is quite different from that of inner product spaces. For example, not any closed submodule of an inner product $C^{*}$-module is complemented; a bounded $C^*$-linear operator on an inner product $C^{*}$-module may not have an adjoint operator. We refer the reader to \cite{lan, Man} for more information on the basic theory of Hilbert $C^{*}$-modules.


\section{Linear mappings preserving some angles in inner product spaces}

We start our work with the following proposition. This result is of independent interest.

\begin{proposition}\label{lm.28}
Let $\theta\in(0, \pi)$, $\mathcal{X}$ be a real inner product space and $x, y$ be two independent vectors in $\mathcal{X}$ such that $\theta\neq\widehat{(x, y)}$. Then the following conditions are equivalent:
\begin{itemize}
\item[(i)] $\|x\|=\|y\|$;
\item[(ii)] there exists a nonzero scalar $\lambda$ such that
$$x+\lambda y\underset{\theta}{\angle} y \quad \mbox{and} \quad y+\lambda x\underset{\theta}{\angle} x.$$
\end{itemize}
Moreover, $\lambda$ is uniquely determined by $$\lambda=-\frac{\langle x, y\rangle-\cot\theta\sqrt{\|x\|^4-{\langle x, y\rangle}^2}}{\|x\|^2}.$$
\end{proposition}

\begin{proof}
(i)$\Rightarrow$(ii) We may suppose that $\langle x, y\rangle \neq 0$. Assume that $\|x\|^2=\|y\|^2=\alpha$. Put
$$\lambda:=-\frac{\langle x, y\rangle-\cot\theta\sqrt{\alpha^2-{\langle x, y\rangle}^2}}{\alpha}.$$
Note that $\lambda\neq0$. Indeed, if $\lambda=0$, then $\langle x, y\rangle=\cot\theta\sqrt{\alpha^2-{\langle x, y\rangle}^2}$. So, $\cos\theta=\frac{\langle x, y\rangle}{\|x\|\|y\|}$, which is impossible. We have
$$\|x+\lambda y\|^2=\|x\|^2+\lambda^2\|y\|^2+2\lambda\langle x, y\rangle=(\lambda^2+1)\alpha+2\lambda\langle x, y\rangle$$
and
$$\|y+\lambda x\|^2=\|y\|^2+\lambda^2\|x\|^2+2\lambda\langle x, y\rangle=(\lambda^2+1)\alpha+2\lambda\langle x, y\rangle.$$
Thus we get
\begin{align}\label{id..25}
\|x+\lambda y\|&=\|y+\lambda x\|\nonumber
\\&=\sqrt{(\lambda^2+1)\alpha+2\lambda\langle x, y\rangle}\nonumber
\\&=\Big[\frac{{\langle x, y\rangle}^2}{\alpha}-2\frac{\langle x, y\rangle\cot\theta\sqrt{\alpha^2-{\langle x, y\rangle}^2}}{\alpha}+\alpha\cot^2\theta\nonumber
\\&\hspace{1cm}-\frac{{\langle x, y\rangle}^2\cot^2\theta}{\alpha}+\alpha-\frac{2{\langle x, y\rangle}^2}{\alpha}+2\frac{\langle x, y\rangle\cot\theta\sqrt{\alpha^2-{\langle x, y\rangle}^2}}{\alpha}\Big]^\frac{1}{2}\nonumber
\\&=\frac{\sqrt{\alpha^2-{\langle x, y\rangle}^2}}{\sqrt{\alpha}\sin\theta}.
\end{align}
Since $\|x\|=\|y\|$, we have
\begin{align}\label{id..5}
\langle y+\lambda x, x\rangle&=\langle x+\lambda y, y\rangle\nonumber
\\&=\langle x, y\rangle+\lambda\|y\|^2\nonumber
\\&=\langle x, y\rangle-\left(\frac{\langle x, y\rangle-\cot\theta\sqrt{\alpha^2-{\langle x, y\rangle}^2}}{\alpha}\right)\alpha\nonumber
\\&=\cot\theta\sqrt{\alpha^2-{\langle x, y\rangle}^2}.
\end{align}
Notice that since $x$ and $y$ are independent, $\sqrt{\alpha^2-{\langle x, y\rangle}^2} \neq 0$. It follows from (\ref{id..25}) and (\ref{id..5}) that
$$\frac{\langle x+\lambda y, y\rangle}{\|x+\lambda y\|\|y\|}=\frac{\langle y+\lambda x, x\rangle}{\|y+\lambda x\|\|x\|}=\frac{\cot\theta\sqrt{\alpha^2-{\langle x, y\rangle}^2}}{\frac{\sqrt{\alpha^2-{\langle x, y\rangle}^2}}{\sqrt{\alpha}\sin\theta}\sqrt{\alpha}}=\cos\theta.$$
Thus $x+\lambda y\underset{\theta}{\angle} y$ and $y+\lambda x\underset{\theta}{\angle} x.$
In addition, the uniqueness of $\lambda$ is concluded from the equality of the cosine of $\widehat{(x+\lambda y, y)}$ and that of $\widehat{(y+\lambda x, x)}$.

\noindent
(ii)$\Rightarrow$(i)\,Suppose (ii) holds. Since $\widehat{(x+\lambda y, y)}=\theta=\widehat{(y+\lambda x, x)}$, we have
\begin{align}\label{id.1}
\frac{\langle x+\lambda y, y\rangle}{\|x+\lambda y\|\|y\|}=\frac{\langle y+\lambda x, x\rangle}{\|y+\lambda x\|\|x\|}.
\end{align}
Therefore,
\begin{align*}
\Big[{\langle x, y\rangle}^2&+\lambda^2\|y\|^4+2\lambda\|y\|^2\langle x, y\rangle\Big]\Big[\|y\|^2+\lambda^2\|x\|^2+2\lambda\langle x, y\rangle\Big]\|x\|^2
\\&=\Big[\langle x+\lambda y, y\rangle\|y+\lambda x\|\|x\|\Big]^2
\\&=\Big[\langle y+\lambda x, x\rangle\|x+\lambda y\|\|y\|\Big]^2\hspace{5cm}(\mbox{by}\,(\ref{id.1}))
\\&=\Big[{\langle x, y\rangle}^2+\lambda^2\|x\|^4+2\lambda\|x\|^2\langle x, y\rangle\Big]\Big[\|x\|^2+\lambda^2\|y\|^2+2\lambda\langle x, y\rangle\Big]\|y\|^2
\end{align*}
A straightforward computation shows that
$$\lambda(\|x\|^2-\|y\|^2)({\langle x, y\rangle}^2-\|x\|^2\|y\|^2)\Big[\lambda(\|x\|^2+\|y\|^2)+2{\langle x, y\rangle}\Big]=0.$$
If $\langle x, y\rangle =0$, then by identity (\ref{id.1}), we arrive at $\|x\|=\|y\|$. So, we may suppose that $\langle x, y\rangle \neq 0$, and two cases occur:

(1) If $\lambda(\|x\|^2+\|y\|^2)+2{\langle x, y\rangle}=0$, then $\lambda=-\frac{2{\langle x, y\rangle}}{\|x\|^2+\|y\|^2}$. Hence by identity (\ref{id.1}), we deduce that
$$\frac{\|x\|^2-\|y\|^2}{\|x+\lambda y\|\|y\|}=\frac{\|y\|^2-\|x\|^2}{\|y+\lambda x\|\|x\|},$$
whence $\|x\|=\|y\|$.

(2) If $\lambda(\|x\|^2-\|y\|^2)({\langle x, y\rangle}^2-\|x\|^2\|y\|^2)=0$, then because of $x$ and $y$ are independent we get $\lambda({\langle x, y\rangle}^2-\|x\|^2\|y\|^2)\neq0$. We conclude that $\|x\|=\|y\|$. Thus, (i) holds true.
\end{proof}

\begin{remark}
The case $\theta = 0$ can take place only for parallel vectors $x,y$ with $\langle x,y\rangle$ positive and $\lambda = 1$, or for $\langle x,y\rangle$ negative and $\lambda = -1$. The case $\theta = \pi$ can take place only for parallel vectors $x,y$ with $\langle x,y\rangle$ negative with $\lambda = 1$, and for $\langle x,y\rangle$ positive with $\lambda = 1$.
\end{remark}

Putting $\theta = \frac{\pi}{2}$ in Proposition \ref{lm.28}, we get the next result.

\begin{proposition}\label{pr.21}
Let $\mathcal{X}$ be a real inner product space and $x, y$ be two nonorthogonal, nonparallel vectors in $\mathcal{X}$. The following conditions are equivalent:
\begin{itemize}
\item[(i)] $\|x\|=\|y\|$;
\item[(ii)] there exists a unique nonzero scalar $\lambda$, given by $\lambda=-\frac{\langle x, y\rangle}{\|x\|^2}$, such that
$$x+\lambda y\perp y \quad \mbox{and} \quad y+\lambda x\perp x.$$
\end{itemize}
\end{proposition}

\begin{corollary}\label{cr.21.1}
Let $\mathcal{X}$ be a real inner product space and $x, y$ be two nonorthogonal vectors in $\mathcal{X}$. The following conditions are equivalent:
\begin{itemize}
\item[(i)] there exists a unique nonzero scalar $\lambda$, given by $\lambda=-\frac{\langle x, y\rangle}{\|x\|^2}$, such that $x+\lambda y\perp y$ and $y+\lambda x\perp x$;
\item[(ii)] there exists a unique scalar $\mu$ of modulus one, given by $\mu=\frac{\langle x, y\rangle}{|\langle x, y\rangle|}$, such that $x+\mu y\perp x-\mu y$.
\end{itemize}
Furthermore, each one of the assertions above implies $\widehat{(x, x+y)}=\widehat{(y, y+x)}$.
\end{corollary}

\begin{proof}
(i)$\Rightarrow$(ii) Set $\mu=\frac{\langle x, y\rangle}{|\langle x, y\rangle|}$. So $|\mu|=1$. By Proposition \ref{pr.21}, $\|x\|=\|y\|$. Now simple computations show that $x+\mu y\perp x-\mu y$.\\
(ii)$\Rightarrow$(i) Since $|\mu|=1$ and $x+\mu y\perp x-\mu y$, we conclude that $\|x\|=\|y\|$. Therefore, (i) holds by Proposition \ref{pr.21}.\\
Assume, that there exists a nonzero scalar $\lambda$ such that $x+\lambda y\perp y$ and $y+\lambda x\perp x$. So by Proposition \ref{pr.21}, $\|x\|=\|y\|$. Therefore
$$\frac{\langle x, x+y\rangle}{\|x\|\|x+y\|}=\frac{\|x\|}{\|x+y\|}+\frac{\langle x, y\rangle}{\|x\|\|x+y\|}=\frac{\|y\|}{\|y+x\|}+\frac{\langle y, x\rangle}{\|y\|\|y+x\|}=\frac{\langle y, y+x\rangle}{\|y\|\|y+x\|},$$
whence $\widehat{(x, x+y)}=\widehat{(y, x+y)}$.
\end{proof}

We are now in a position to establish the main result of this section.

\begin{theorem}\label{th.23}
Let $\mathcal{X}$ and $\mathcal{Y}$ be two real inner product spaces. For a nonzero linear mapping $T:\mathcal{X}\longrightarrow \mathcal{Y}$ the following statements are equivalent:
\begin{itemize}
\item[(i)] there exists $\gamma>0$ such that $\|Tx\|=\gamma\|x\|$ for all $x\in \mathcal{X}$, i.e., $T$ is a similarity;
\item[(ii)] $T$ is injective and $\frac{\langle Tx, Ty\rangle}{\|Tx\|\|Ty\|}=\frac{\langle x, y\rangle}{\|x\|\|y\|}$ for all $x, y\in \mathcal{X}\smallsetminus\{0\}$;
\item[(iii)] $x\perp y\, \Leftrightarrow \,Tx\perp Ty$ for all $x, y\in \mathcal{X}$, i.e., $T$ is strongly orthogonality preserving;
\item[(iv)] $\|x\|=\|y\|\, \Leftrightarrow \,\|Tx\|=\|Ty\|$ for all $x, y\in \mathcal{X}$;
\item[(v)] $\|x\|=\|y\|\, \Rightarrow \,\|Tx\|=\|Ty\|$ for all $x, y\in \mathcal{X}$;
\item[(vi)] $\|x\|\leq\|y\|\, \Rightarrow \,\|Tx\|\leq\|Ty\|$ for all $x, y\in \mathcal{X}$;
\item[(vii)] $x\perp y\, \Rightarrow \,Tx\perp Ty$ for all $x, y\in \mathcal{X}$, i.e., $T$ is orthogonality preserving.
\end{itemize}
\end{theorem}

\begin{proof}
(i)$\Rightarrow$(ii) Clearly $T$ is injective. In addition, for all $x, y\in \mathcal{X}\smallsetminus\{0\}$ we have
\begin{align*}
\frac{\langle Tx, Ty\rangle}{\|Tx\|\|Ty\|}&=\frac{\frac{1}{4}\Big[\|T(x+y)\|^2-\|T(x-y)\|^2\Big]}{(\gamma\|x\|)(\gamma\|y\|)}
\\&=\frac{\frac{1}{4}\Big[\gamma^2\|x+y\|^2-\gamma^2\|x-y\|^2\Big]}{\gamma^2\|x\|\|y\|}
\\&=\frac{\langle x, y\rangle}{\|x\|\|y\|}.
\end{align*}
(ii)$\Rightarrow$(iii) This implication is trivial.\\
(iii)$\Rightarrow$(iv) $$\|x\|=\|y\| \Leftrightarrow x+y \perp x-y \Leftrightarrow Tx+Ty\perp Tx-Ty \Leftrightarrow \|Tx\|=\|Ty\| \,.$$
(iv)$\Rightarrow$(v) This implication is trivial.\\
(v)$\Rightarrow$(vi) Let $\|x\|\leq\|y\|$. Since $\|x\| = \left\|\frac{\|x\|}{\|y\|}y\right\|$, so with condition (v) we obtain
$$\|Tx\| = \left\|T(\frac{\|x\|}{\|y\|}y)\right\| = \frac{\|x\|}{\|y\|}\,\|Ty\|\leq \,\|Ty\|.$$
(vi)$\Rightarrow$(vii) Suppose (vi) holds and $x\perp y$. Hence, $\|x\|\leq\|x + ry\|$ for all $r\in\mathbb{R}$. Thus $\|Tx\|\leq\|Tx + rTy\|$ for all $r\in\mathbb{R}$, and therefore, $Tx\perp Ty$.\\
(vii)$\Rightarrow$(i) Let $x_0\in \mathcal{X}$ be fixed such that $\|x_0\|=1$. So $Tx_0\neq0$. Indeed, if $Tx_0 = 0$, then for any $0 \neq z\in \mathcal{X}$ we have $\langle \frac{z}{\|z\|} + x_0, \frac{z}{\|z\|} - x_0\rangle = 0$. The orthogonality preserving property of $T$ infers $\|Tz\| = 0$, whence $T = 0$, which is a contradiction.\\
Since $\langle \frac{w}{\|w\|} + x_0, \frac{w}{\|w\|} - x_0\rangle = 0$ for every $w\in \mathcal{X}\smallsetminus\{0\}$, we get $\|Tw\| = \|Tx_0\|\,\|w\|$ by assumption and, hence, assertion (i).
\end{proof}

The following example shows that Theorem \ref{th.23} fails if the supposition of linearity is dropped.

\begin{example}
Let $\mathcal{X}$ be a real inner product space. For the nonlinear mapping $T:\mathcal{X}\longrightarrow \mathcal{X}$ defined by $T(x) = \|x\|^2 x$ we have $\|x\|=\|y\|\, \Rightarrow \,\|Tx\|=\|Ty\|$ for all $x, y\in \mathcal{X}$, but $T$ is clearly not a similarity.
\end{example}

\begin{corollary}\label{co.27}
Let $\mathcal{X}$ be a real vector space equipped with two inner products ${\langle ., .\rangle}_1$ and ${\langle ., .\rangle}_2$ generating the norms ${\|.\|}_1$, ${\|.\|}_2$ and orthogonality relations $\perp_1$, $\perp_2$, respectively. Then the following conditions are equivalent:
\begin{itemize}
\item[(i)]$x\perp_1y\, \Rightarrow \,x\perp_2y \qquad (x, y\in \mathcal{X})$;
\item[(ii)] $x+z\perp_1y+z\, \Rightarrow \,x+z\perp_2y+z \qquad (x, y, z\in \mathcal{X})$;
\item[(iii)] ${\|x\|}_1={\|y\|}_1\, \Rightarrow \,{\|x\|}_2={\|y\|}_2 \qquad (x, y\in \mathcal{X})$;
\item[(iv)] ${\|x\|}_1\leq{\|y\|}_1\, \Rightarrow \,{\|x\|}_2\leq{\|y\|}_2 \qquad (x, y\in \mathcal{X})$;
\item[(v)] ${\|x+z\|}_1={\|y+z\|}_1\, \Rightarrow \,{\|x+z\|}_2={\|y+z\|}_2 \qquad (x, y, z\in \mathcal{X})$;
\item[(vi)] ${\|x+z\|}_1\leq{\|y+z\|}_1\, \Rightarrow \,{\|x+z\|}_2\leq{\|y+z\|}_2 \qquad (x, y, z\in \mathcal{X})$;
\item[(vii)] $\frac{{\langle x, y\rangle}_1}{{\|x\|}_1{\|y\|}_1}=\frac{{\langle x, y\rangle}_2}{{\|x\|}_2{\|y\|}_2} \qquad (x, y\in \mathcal{X})$;
\item[(ix)] $\frac{{\langle x+z, y+z\rangle}_1}{{\|x+z\|}_1{\|y+z\|}_1}=\frac{{\langle x+z, y+z\rangle}_2}{{\|x+z\|}_2{\|y+z\|}_2} \qquad (x, y, z\in \mathcal{X})$;
\item[(x)] ${\|x\|}_2=\gamma{\|y\|}_1 \qquad (x, y\in \mathcal{X})$ with some $\gamma>0$;
\item[(xi)] ${\langle x, y\rangle}_2=\gamma^2{\langle x, y\rangle}_1 \qquad (x, y\in \mathcal{X})$ with some $\gamma>0$.
\end{itemize}
\end{corollary}

\begin{proof}
Take $\mathcal{X}=\mathcal{Y}$ and $T=id:(\mathcal{X}, {\langle ., .\rangle}_1)\longrightarrow (\mathcal{X}, {\langle ., .\rangle}_2)$, and apply Theorem \ref{th.23}.
\end{proof}

By the equivalence (i)$\Leftrightarrow$(ii) of Theorem \ref{th.23}, if $\theta\in(0, \pi)$ and $T:\mathcal{X}\longrightarrow \mathcal{Y}$ is a similarity, then $x\underset{\theta}{\angle} y\, \Leftrightarrow \,Tx\underset{\theta}{\angle} Ty$ for all $x, y\in \mathcal{X}$. The following theorem is an extension of Theorem \ref{th.23}.

\begin{theorem}\label{th.29}
Let $\theta\in(0, \pi)$ and $\mathcal{X}$ and $\mathcal{Y}$ be two real inner product spaces. If $T:\mathcal{X}\longrightarrow \mathcal{Y}$ is an injective nonzero linear mapping with the following properties:
\begin{itemize}
\item[(i)] $x\underset{\theta}{\angle} y\, \Leftrightarrow \,Tx\underset{\theta}{\angle} Ty$ for all $x, y\in \mathcal{X}$,
\item[(ii)] If $\|x\|=\|y\|$ with $x\underset{\theta}{\angle} y$, then $\|Tx\|=\|Ty\|$ for all $x, y\in \mathcal{X}$,
\end{itemize}
then $T$ is a similarity.
\end{theorem}

\begin{proof}
Assume, $x, y\in \mathcal{X}$ and $\|x\|=\|y\|$. If $x\underset{\theta}{\angle} y$, then by (ii) we have $\|Tx\|=\|Ty\|$.
Consider the case when $\theta\neq\widehat{(x, y)}$. If $x$ and $y$ are linearly dependent, we get $\|Tx\|=\|Ty\|$ since $\|x\|=\|y\|$. Thus, in the case that $x$ and $y$ are independent, there exists a nonzero scalar $\lambda$ such that
$x+\lambda y\underset{\theta}{\angle} y $ and $y+\lambda x\underset{\theta}{\angle} x$ by Proposition \ref{lm.28}. By (i), we get $Tx+\lambda Ty\underset{\theta}{\angle} Ty $ and $Ty+\lambda Tx\underset{\theta}{\angle} Tx$. Since $T$ is injective, $Tx$ and $Ty$ are non-parallel.The condition $\theta\neq\widehat{(Tx, Ty)}$ and Proposition \ref{lm.28} imply $\|Tx\|=\|Ty\|$.\\
Thus, $\|x\|=\|y\|\, \Rightarrow \,\|Tx\|=\|Ty\|$ for all $x, y\in \mathcal{X}$. By the equivalence (i)$\Leftrightarrow$(v) of Theorem \ref{th.23}, $T$ is a similarity.
\end{proof}

\begin{remark}\label{re.210}
Note, that for $\theta=0, \pi$, i.e. if $x\parallel y$, $x, y$ are linearly dependent. For every linear mapping $T$, therefore $Tx$ and $Ty$ are linearly dependent too. Thus we have $x\parallel y\, \Rightarrow \,Tx\parallel Ty$ for any liner mapping $T$ and any $x, y\in \mathcal{X}$.
\end{remark}


\section{Orthogonality preserving mappings in inner product $C^{*}$-modules}

The notion of orthogonality in an arbitrary normed space $(\mathcal{V}, \|.\|)$ may be introduced in various ways. Many mathematicians have introduced different types of orthogonality. Roberts \cite{R} introduced the first type of orthogonality : $x \in \mathcal{V}$ is said to be orthogonal in the sense of Roberts to $y \in \mathcal{V}$ if $\|x+ty\|=\|x-ty\|$ for all $t\in \mathbb{R}$.
Later Birkhoff \cite{B} introduced one of the most important types of orthogonality: $x$ is said to be Birkhoff orthogonal to $y$ if $\|x\|\leq\|x + \lambda y\|$ for all $\lambda\in \mathbb{C}.$ For inner product spaces, these definitions are equivalent to the usual definition of orthogonality. Characterizations of the Birkhoff orthogonality in the framework of Hilbert $C^{*}$-modules were obtained in \cite{A.R, B.G}. There are other ways to generalize the Roberts orthogonality. In the following auxiliary lemma we state some facts involving the ``$C^*$-valued norm'' $|\cdot|$ on an inner product $\mathscr{A}$-module, see also \cite[Proposition 2.1]{I.T}.

\begin{lemma}\label{lm.211}
Let $\mathscr{E}$ be an inner product $\mathscr{A}$-module and $x, y\in \mathscr{E}$. The following statements are mutually equivalent:
\begin{itemize}
\item[(i)] $\langle x, y\rangle=0$;
\item[(ii)] $\langle xb, ya\rangle=0$ for all $a, b\in \mathscr{A}$;
\item[(iii)] $|x+ya|=|x-ya|$ for all $a\in \mathscr{A}$;
\item[(iv)] $|x+\lambda y|=|x-\lambda y|$ for all $\lambda\in \mathbb{C}$;
\end{itemize}
\end{lemma}
\begin{proof}
(i)$\Rightarrow$(ii) The implication is trivial.

(ii)$\Rightarrow$(iii) Let $(e_i)_{i\in I}$ be an approximate unit for $\mathscr{A}$. Since (ii) is valid we have
\begin{align*}
\|\langle x, y\rangle\|& = \|\langle xe_i, ye_i\rangle - \langle x, y\rangle\|
\\& \leq \|e_i\|\|\langle x, y\rangle e_i - \langle x, y\rangle\| + \|e_i\langle x, y\rangle - \langle x, y\rangle\|.
\end{align*}
By taking limit we get $\langle x, y\rangle = 0$. Thus, $|x+ya|=\Big[|x|^2+|ya|^2\Big]^\frac{1}{2}=|x-ya|$ for all $a\in \mathscr{A}$.

(iii)$\Rightarrow$(iv) We may assume that $y\neq0$. Let $a =\langle y, x\rangle\in\mathscr{A}$. Then
\begin{align*}
|x|^2+2\langle x, y\rangle\langle y, x\rangle+|y\langle y, x\rangle|^2&=|x|^2+\langle x, y\rangle a+a^*\langle y, x\rangle+|ya|^2
\\&=|x+ya|^2
\\&=|x-ya|^2
\\&=|x|^2-\langle x, y\rangle a-a^*\langle y, x\rangle+|ya|^2
\\&=|x|^2-2\langle x, y\rangle\langle y, x\rangle+|y\langle y, x\rangle|^2,
\end{align*}
which implies $\langle x, y\rangle\langle y, x\rangle=0$ and so $\langle x, y\rangle=0$. Thus, $|x+\lambda y|= \Big[|x|^2 + |\lambda|^2 |y|^2\Big]^\frac{1}{2} = |x-\lambda y|$ for all $\lambda\in \mathbb{C}$.

(iv)$\Rightarrow$(i) Suppose that (iv) holds. Therefore $\lambda\langle y, x\rangle + \overline{\lambda}\langle x, y\rangle =0$ for all $\lambda\in \mathbb{C}$. So, for $\lambda = 1, -i$ we obtain $\langle y, x\rangle + \langle x, y\rangle =0$ and $-i\langle y, x\rangle +i \langle x, y\rangle =0$. Hence $\langle x, y\rangle=0$.
\end{proof}

\begin{remark} In Theorem 3.1 of \cite{A.R.1} three further equivalent inequality conditions to assertion (i) above have been found: $\,$ (vi) $|x|^2 \leq |x + ya|^2$ for all $a \in {\mathscr{A}}$; $\,$ (vii) $|x|^2 \leq |x + \lambda y|^2$ for all $\lambda \in {\mathbb C}$; $\,$ (viii) $|x| \leq |x + ya|$ for all $a \in {\mathscr{A}}$. The authors of \cite{A.R.1} conjecture that assertion $\,$ (ix)
$| x | \leq |x + \lambda y|$ for all $\lambda \in {\mathbb C}$ $\,$ might be also equivalent to assertion (i). To verify this to the affirmative, one would have to prove the equivalence of (i) and (ix) merely for (atomic) von Neumann factors considered as Hilbert $C^*$-modules over themselves. Indeed, any full Hilbert $C^*$-module $\mathscr{E}$ can be isometrically embedded into its linking ($C^*$-)algebra as constructed by Brown, Green and Rieffel in \cite{BGR}. The $C^*$-algebra of coefficients $\mathscr{A}$, the Hilbert $\mathscr{A}$-module $\mathscr{E}$ and the $C^*$-algebra of all ``compact'' operators on $\mathscr{E}$ become full corners of the linking algebra carried by two orthogonal projections from the multiplier $C^*$-algebra of the linking algebra which sum up to the identity. The module action becomes simple multiplication. Both the assertions (i) and (ix) survive this isometric embedding, so their equivalence would have to be shown for $C^*$-algebras only which are considered as Hilbert $C^*$-modules over themselves.

The second trick is due to Akemann, see \cite[p. 278]{A1}, \cite[p. I]{A2}: Any $C^*$-algebra can be $*$-isometri\-cally embedded first in its bidual linear space, a von Neumann algebra, by classical $*$-representation theory of $C^*$-algebras. Further, the restriction of that $*$-isometric embedding to the atomic part of the bidual von Neumann algebra is again a $*$-isomorphism simply multiplying by the central carrier projection of the atomic part. This embedding also preserves both the assertions (i) and (ix). Since the center of the atomic von Neumann algebras is an atomic commutative von Neumann algebra one could consider the equivalence of these assertions (i) and (ix) for any atomic projection of it separately, recomposing the entire picture afterwards. This outlines, ay help one to prove or disprove the conjecture at least for the von Neumann type I factors.
\end{remark}

Recall that a linear mapping $T:\mathscr{E}\longrightarrow \mathscr{F}$, where $\mathscr{E}$ and $\mathscr{F}$ are inner product $\mathscr{A}$-modules, is said to be orthogonality preserving if $\langle x, y\rangle=0\, \Rightarrow \,\langle Tx, Ty\rangle=0$ for all $x, y\in \mathscr{E}$.

\begin{theorem}\label{th.212.5}
Let $\mathscr{E}$ and $\mathscr{F}$ be two inner product $\mathscr{A}$-modules and let $T:\mathscr{E}\longrightarrow \mathscr{F}$ be a nonzero linear mapping. If $|x|=|y|\, \Rightarrow \,|Tx|=|Ty|$ for all $x, y\in \mathscr{E}$, then $T$ is orthogonality preserving.
\end{theorem}

\begin{proof}
Let $x, y\in \mathscr{E}$ such that $\langle x, y\rangle=0$. By the equivalence (i)$\Leftrightarrow$(iv) of Lemma \ref{lm.211} we have $|x+\lambda y|=|x-\lambda y|$ for all $\lambda\in \mathbb{C}$. We get $|Tx + \lambda Ty|=|Tx - \lambda Ty|$ for all $\lambda\in \mathbb{C}$ by assumption. So, $\lambda\langle Tx, Ty\rangle - \overline{\lambda}\langle Ty, Tx\rangle =0$. Putting $\lambda = 1, i$, we obtain $\langle Tx, Ty\rangle + \langle Ty, Tx\rangle =0$ and $\langle Tx, Ty\rangle - \langle Ty, Tx\rangle =0$. Therefore, $\langle Tx, Ty\rangle=0$ what forces $T$ to be orthogonality preserving.
\end{proof}

\begin{theorem}\label{th.212.5.1}
Let $\mathscr{E}$ and $\mathscr{F}$ be two inner product $\mathscr{A}$-modules and let $T:\mathscr{E}\longrightarrow \mathscr{F}$ be a nonzero $\mathscr{A}$-linear mapping. If $|x|\leq|y|\, \Rightarrow \,|Tx|\leq|Ty|$ for all $x, y\in \mathscr{E}$, then $T$ is orthogonality preserving.
\end{theorem}

\begin{proof}
Suppose, $\langle x, y\rangle=0$. Thus, $|x|\leq\Big[|x|^2 + \left|y\frac{\langle Ty, Tx\rangle}{\|Ty\|^2}\right|^2\Big]^\frac{1}{2} = \left|x - y\frac{\langle Ty, Tx\rangle}{\|Ty\|^2}\right|$. By assumption we get
\begin{align*}
|Tx| & \leq \left|Tx - Ty\frac{\langle Ty, Tx\rangle}{\|Ty\|^2}\right|
\\ & = \Big[|Tx|^2-\frac{2}{\|Ty\|^2}\langle Tx, Ty\rangle\langle Ty, Tx\rangle + \frac{1}{\|Ty\|^4}\langle Tx, Ty\rangle\langle Ty, Ty\rangle\langle Ty, Tx\rangle\Big]^\frac{1}{2}
\\&\leq\Big[|Tx|^2-\frac{2}{\|Ty\|^2}\langle Tx, Ty\rangle\langle Ty, Tx\rangle + \frac{1}{\|Ty\|^4}\langle Tx, Ty\rangle\|Ty\|^2\langle Ty, Tx\rangle\Big]^\frac{1}{2}
\\&=\Big[|Tx|^2-\frac{1}{\|Ty\|^2}\langle Tx, Ty\rangle\langle Ty, Tx\rangle\Big]^\frac{1}{2}
\\&\leq|Tx|.
\end{align*}
Therefore, $\langle Tx, Ty\rangle\langle Ty, Tx\rangle=0$ and $\langle Tx, Ty\rangle=0$. As a consequence, $T$ is orthogonality preserving.
\end{proof}

\begin{remark}
Consider the $C^*$-algebra $\mathbb{M}_2(\mathbb{C})$ as a Hilbert $C^*$-module over itself. If $T:\mathbb{M}_2(\mathbb{C})\longrightarrow \mathbb{M}_2(\mathbb{C})$ is a linear mapping, not necessarily an $\mathbb{M}_2(\mathbb{C})$-linear mapping, such that $|A|\leq|B|\, \Rightarrow \,|TA|\leq|TB|$ for all $A, B\in \mathbb{M}_2(\mathbb{C})$, then $T$ is orthogonality preserving. Indeed, let $A, B\in \mathbb{M}_2(\mathbb{C})$ such that $A^\ast B=0$. Thus
$$|A|\leq \Big[|A|^2+|\mu B|^2\Big]^\frac{1}{2} = |A + \mu B|,$$
for all $\mu\in \mathbb{C}$. Hence, we get $|TA|\leq|TA + \mu TB|$ for all $\mu\in \mathbb{C}$ by assumption. By \cite[Proposition 3.6]{A.R.1} we obtain $(TA)^\ast TB=0$, and $T$ is orthogonality preserving.
\end{remark}

In the following theorem we present a variant of Theorem \ref{th.23} for the context of $C^{*}$-modules.

\begin{theorem}\label{th.212}
Let $\mathscr{E}$ and $\mathscr{F}$ be two inner product $\mathscr{A}$-modules. For a nonzero $\mathscr{A}$-linear mapping $T:\mathscr{E}\longrightarrow \mathscr{F}$ the following statements are equivalent:
\begin{itemize}
\item[(i)] there exists $\gamma>0$ such that $\|Tx\|=\gamma\|x\|$ for all $x\in \mathscr{E}$, i.e., $T$ is a similarity;\\
\item[(ii)] $T$ is injective and $\frac{\langle Tx, Ty\rangle}{\|Tx\|\|Ty\|}=\frac{\langle x, y\rangle}{\|x\|\|y\|}$ for all $x, y\in \mathscr{E}\smallsetminus\{0\}$.
\end{itemize}
Furthermore, each one of the assertions above implies:
\begin{itemize}
\item[(iii)] $\langle x, y\rangle=0\, \Leftrightarrow \,\langle Tx, Ty\rangle=0$ for all $x, y\in \mathscr{E}$, i.e., $T$ is strongly orthogonality preserving;\\
\item[(iv)] $|x|=|y|\, \Leftrightarrow \,|Tx|=|Ty|$ for all $x, y\in \mathscr{E}$;\\
\item[(v)] $|x|\leq|y|\, \Leftrightarrow \,|Tx|\leq|Ty|$ for all $x, y\in \mathscr{E}$.
\end{itemize}
\end{theorem}

\begin{proof}
(i)$\Rightarrow$(ii) Clearly, $T$ is injective. For all $a\in\mathscr{A}$ and $x\in \mathscr{E}$
\begin{align}\label{id.2}
\Big\||Tx|a\Big\|&=\sqrt{\|(|Tx|a)^*(|Tx|a)\|}=\sqrt{\|a^*|Tx|^2a\|}\nonumber
\\&=\sqrt{\|a^*\langle Tx, Tx\rangle a\|}=\sqrt{\|\langle T(xa), T(xa)\rangle\|}\nonumber
\\&=\|T(xa)\|=\|\gamma(xa)\|\nonumber
\\&=\sqrt{\|\langle \gamma(xa), \gamma(xa)\rangle\|}=\sqrt{\|(\gamma|x|a)^*(\gamma|x|a)\|}\nonumber
\\&=\Big\|(\gamma|x|)a\Big\|.
\end{align}
Since $|Tx|$ and $\gamma|x|$ are positive, (\ref{id.2}) implies $|Tx|=\gamma|x|$. Now, for all $x, y\in \mathscr{E}\smallsetminus\{0\}$ we obtain
\begin{align*}
\frac{\langle Tx, Ty\rangle}{\|Tx\|\|Ty\|}&=\frac{\frac{1}{4}\sum_{k=0}^3i^k|T(x+i^ky)|^2}{(\gamma\|x\|)(\gamma\|y\|)}
\\&=\frac{\frac{1}{4}\sum_{k=0}^3i^k\gamma^2|x+i^ky|^2}{\gamma^2\|x\|\|y\|}
\\&=\frac{\langle x, y\rangle}{\|x\|\|y\|}.
\end{align*}
(ii)$\Rightarrow$(i) Fix $x_0\in \mathscr{E}$ with $\|x_0\|=1$ and set $\gamma=\|Tx_0\|$. Since $T$ is injective, so $\gamma>0$. For every $x\in \mathscr{E}$, if $x$ and $x_0$ are linearly dependent, then obviously $\|Tx\|=\gamma\|x\|$. Assume now, that $x$ and $x_0$ are linearly independent. By (ii) we get
\begin{align*}
\Big\langle z, (x+x_0)\frac{\|Tz\|\|T(x+x_0)\|}{\|z\|\|(x+x_0)\|}\Big\rangle&=\langle Tz, T(x+x_0)\rangle
\\&=\langle Tz, Tx\rangle+\langle Tz, Tx_0\rangle
\\&=\Big\langle z, x\frac{\|Tz\|\|Tx\|}{\|z\|\|x\|}\Big\rangle+\Big\langle z, x_0\frac{\|Tz\|\|Tx_0\|}{\|z\|\|x_0\|}\Big\rangle
\\&=\Big\langle z, x\frac{\|Tz\|\|Tx\|}{\|z\|\|x\|}+x_0\frac{\|Tz\|\|Tx_0\|}{\|z\|}\Big\rangle,
\end{align*}
for all $z\in \mathscr{E}\smallsetminus\{0\}$. Whence
$$(x+x_0)\frac{\|T(x+x_0)\|}{\|(x+x_0)\|}=x\frac{\|Tx\|}{\|x\|}+x_0\|Tx_0\|,$$
or equivalently,
$$x\Big(\frac{\|T(x+x_0)\|}{\|(x+x_0)\|}-\frac{\|Tx\|}{\|x\|}\Big)=x_0\Big(\|Tx_0\|-\frac{\|T(x+x_0)\|}{\|(x+x_0)\|}\Big).$$
The equality $\|Tx\|=\|Tx_0\|\|x\|=\gamma\|x\|$ follows.\\
(ii)$\Rightarrow$(iii) This implication is trivial by using (i).\\
(i)$\Rightarrow$(iv), (v)
Assume, (i) holds. As in the proof of the implication (i)$\Rightarrow$(ii), we get $|Tx|=\gamma|x|$ for some $\gamma>0$ and all $x\in \mathscr{E}$, what shows the implications.
\end{proof}

The following example shows that conditions (iii)-(v) are not equivalent to conditions (i)-(ii) in general.

\begin{example}\label{xe.212.1}
Let $\Omega $ be a locally compact
Hausdorff space. Let us take $\mathscr{E} = \mathscr{F} = C_0(\Omega)$, the $C^*$-algebra of all continuous complex-valued functions vanishing at infinity on $\Omega$. For a nonzero function $f_0 \in C_0(\Omega)$, suppose that $T:C_0(\Omega)\longrightarrow C_0(\Omega)$ is given by $T(g) = f_0g$. Obviously
$T$ is $C_0(\Omega)$-linear and satisfies conditions (iii)-(v) but need not satisfies conditions (i)-(ii). Indeed, if there exists $\gamma>0$ such that $\|T(g)\|=\gamma\|g\|$ for all $g\in C_0(\Omega)$, then $\frac{1}{\gamma^2}\overline{f_0}f_0h = h$ for all $h\in C_0(\Omega)$ and hence, $\frac{1}{\gamma^2}\overline{f_0}f_0$ is the identity in $C_0(\Omega)$, which is a contradiction (see \cite[Example 2.4]{I.T}).
\end{example}

\noindent
Recall that a linear mapping $T:\mathscr{E}\longrightarrow \mathscr{F}$, where $\mathscr{E}$ and $\mathscr{F}$ are inner product $\mathscr{A}$-modules, is called local if
$$xa = 0\, \Rightarrow \,(Tx)a = 0 \qquad (a\in\mathscr{A}, x\in \mathscr{E}).$$
Examples of local mappings include multiplication and differential operators. Note that every $\mathscr{A}$-linear mapping is local, but the converse is not true, in general (take linear differential operators into account). Moreover, every bounded local mapping between inner product modules is $\mathscr{A}$-linear \cite[Proposition A.1]{L.N.W.1}.

\begin{lemma}\label{le.212.9}\cite[Theorem 3.1]{I.T}
Let $\mathscr{E}$ and $\mathscr{F}$ be two inner product $\mathscr{A}$-modules such that $\mathbb{K}(\mathscr{H})\subseteq\mathscr{A}\subseteq\mathbb{B}(\mathscr{H})$. Suppose that $T:\mathscr{E}\longrightarrow \mathscr{F}$ is a nonzero orthogonality preserving $\mathscr{A}$-linear map. Then there exists a positive number $\gamma$ such that
\begin{align}\label{id.3.1}
\langle Tx, Ty\rangle = \gamma \langle x, y\rangle
\end{align}
for all $x, y\in \mathscr{E}$.
\end{lemma}

Note that the assumption of $\mathscr{A}$-linearity, even in the case $\mathscr{A} = \mathbb{K}(\mathscr{H})$, is necessary in Lemma \ref {le.212.9} as can be seen from the following example.

\begin{example}
Let $\mathscr{H}$ be a Hilbert space such that $\dim\mathscr{H} = \infty $ and $\mathscr{H}_*= \mathscr{H}$ as an additive group, but define a new scalar multiplication on $\mathscr{H}_*$ by setting $\lambda \cdot x = \overline{\lambda}x$, and a new inner product by setting $\langle x| y\rangle_* = \langle y| x\rangle$. Then $\mathscr{H}_*$ equipped with the operations
$$\langle x, y\rangle := x\otimes y \qquad \mbox{and} \qquad x\cdot S : = S^\ast x \quad (x, y \in \mathscr{H}_* , S \in \mathbb{K}(\mathscr{H}))$$
is an inner product $\mathbb{K}(\mathscr{H})$-module. If $T:\mathscr{H}_*\longrightarrow \mathscr{H}_*$ is any unbounded linear map, then $T$ preserves orthogonality (namely, if $\langle x, y\rangle = x\otimes y = 0$, then $x = 0$ or $y = 0$. So $\langle Tx, Ty\rangle = Tx\otimes Ty = 0$), but $T$ obviously does not satisfy (\ref{id.3.1}).
\end{example}

The next result is a consequence of \cite[Corollary 3.2]{L.N.W.2} but we prove it for the sake of completeness.

\begin{theorem}\label{th.213}
Let $\mathscr{E}$ and $\mathscr{F}$ be two inner product $\mathscr{A}$-modules such that $\mathbb{K}(\mathscr{H})\subseteq\mathscr{A}\subseteq\mathbb{B}(\mathscr{H})$. Suppose that $T:\mathscr{E}\longrightarrow \mathscr{F}$ is a local and nonzero orthogonality preserving map. Then
\begin{itemize}
\item[(i)]$|x|=|y|\, \Leftrightarrow \,|Tx|=|Ty|$ for all $x, y\in \mathscr{E}$;
\item[(ii)]$|x|\leq|y|\, \Leftrightarrow \,|Tx|\leq|Ty|$ for all $x, y\in \mathscr{E}$.
\end{itemize}
\end{theorem}

\begin{proof}
Let $(e_i)_{i\in I}$ and $(f_j)_{j\in J}$ be approximate units for $\mathscr{A}$ and $\mathbb{K}(\mathscr{H})$, respectively. Suppose that $p\in \mathbb{K}(\mathscr{H})$ is a projection. For $x\in \mathscr{E}$ we have
$$xp(1-p)e_i(1-p) = 0 \quad \mbox{and} \quad x(1-p)pe_ip = 0.$$
Since $T$ is local, we obtain
$$T(xp)(1-p)e_i(1-p) = 0 \quad \mbox{and} \quad T(x(1-p))pe_ip = 0.$$
From $\lim\limits_{i}(1-p)e_i(1-p) = 1-p$ and $\lim\limits_{i}pe_ip = p$, we derive
\begin{align*}
(Tx)p &= \Big(T(x(1-p)) + T(xp)\Big)p
\\& = T(x(1-p))p - T(xp)(1-p) + T(xp)
\\& = \lim\limits_{i}T(x(1-p))pe_ip - \lim\limits_{i}T(xp)(1-p)e_i(1-p) + T(xp)
\\& = T(x(1-p))pe_ip - T(xp)(1-p)e_i(1-p) + T(xp)
\\& = T(xp).
\end{align*}
Thus, $T(xa) = (Tx)a$ for all finite rank operators $a\in \mathbb{K}(\mathscr{H})$.\\
Now, for any $x\in \mathscr{E}\cdot \mathbb{K}(\mathscr{H})$, there exist $c\in \mathbb{K}(\mathscr{H})$ and $y\in \mathscr{E}$ such that $x = yc$. Consider the linear mapping $\widetilde{T}:\mathscr{E}\cdot \mathbb{K}(\mathscr{H})\longrightarrow \mathscr{F}\cdot \mathbb{K}(\mathscr{H})$ defined by $\widetilde{T}(x) := T(y)c$. \\
Notice that $\widetilde{T}(x)$ is independent of the decomposition $x = yc$. Therefore,
$$\widetilde{T}(xa) = \widetilde{T}(yca) = T(y)ca = \widetilde{T}(x)a$$
for all $x\in \mathscr{E}$ and all $a\in \mathbb{K}(\mathscr{H})$. Since $(f_j)_{j\in J}$ is an approximate unit for $\mathbb{K}(\mathscr{H})$ it follows from
$\|T(x)f_j - \widetilde{T}(x)\| = \|T(ycf_j) - T(y)c\| = \|T(y)cf_j - T(y)c\|$ that $\lim\limits_{i} T(x)f_j = \widetilde{T}(x)$ for all $x\in \mathscr{E}\cdot \mathbb{K}(\mathscr{H})$ and all $j\in J$. Therefore, if $x_1, x_2\in \mathscr{E}\cdot \mathbb{K}(\mathscr{H})$ with $\langle x_1, x_2\rangle = 0$, then $\langle Tx_1, Tx_2\rangle = 0$ since $T$ is orthogonality preserving, which implies $\langle (Tx_1)f_j, (Tx_2)f_k\rangle = 0$ for all $j, k\in J$. As a consequence, $\langle \widetilde{T}x_1, \widetilde{T}x_2\rangle = 0$, whence $\widetilde{T}$ has to be an orthogonality preserving $\mathbb{K}(\mathscr{H})$- linear map. \\
By Lemma \ref{le.212.9} there exists a positive number $\gamma$ such that $\langle \widetilde{T}x, \widetilde{T}y\rangle = \gamma \langle x, y\rangle$ for all $x, y\in \mathscr{E}\cdot \mathbb{K}(\mathscr{H})$. Thus
$$f_j\langle Tx, Ty\rangle f_k = \langle \widetilde{T}(xf_j), \widetilde{T}(yf_k)\rangle = \gamma \langle xf_j, yf_k\rangle = f_j\gamma \langle x, y\rangle f_k$$
for all $x, y\in \mathscr{E}$ and all $j, k \in J$.
Hence, $\langle Tx, Ty\rangle = \gamma \langle x, y\rangle$ for all $x, y\in \mathscr{E}$, and $|Tx| = \sqrt{\gamma}|x|$ for all $x\in \mathscr{E}$. This actually yields (i) and (ii).
\end{proof}

Combining Theorems \ref{th.212.5.1} and \ref{th.213} we get the next result.

\begin{corollary}
Let $\mathscr{E}$ and $\mathscr{F}$ be two inner product $\mathscr{A}$-modules and $\mathbb{K}(\mathscr{H})\subseteq\mathscr{A}\subseteq\mathbb{B}(\mathscr{H})$. Suppose that $T:\mathscr{E}\longrightarrow \mathscr{F}$ is a nonzero $\mathscr{A}$-linear mapping between inner product $\mathscr{A}$-modules. Then $T$ is orthogonality preserving if and only if $$|x|\leq|y|\, \Rightarrow \,|Tx|\leq|Ty|$$ for all $x, y\in \mathscr{E}$.
\end{corollary}

\end{document}